\documentclass{article}

\usepackage[english]{babel}
\usepackage[utf8x]{inputenc}
\usepackage[T1]{fontenc}
\usepackage{amsmath,amssymb,amsthm,mathtools,mathdots,nicefrac,tikz}

\usepackage[a4paper,top=3cm,bottom=2cm,left=3cm,right=3cm,marginparwidth=1.75cm]{geometry}

\usepackage{amsmath}
\usepackage{amsthm}
\usepackage{amssymb}
\usepackage{graphicx}
\usepackage[colorinlistoftodos]{todonotes}

\theoremstyle{definition} 
\newtheorem{theorem}{Theorem}[section]

\newtheorem{lemma}{Lemma}[section]
\newtheorem{proposition}{Proposition}[section]
\newtheorem{definition}{Definition}[section]

\newtheorem{remark}{Remark}[section]

\newcommand{\Z}{\mathbb{Z}}

\DeclareMathOperator{\Int}{int}
\DeclareMathOperator{\ld}{ld}
\DeclareMathOperator{\lw}{lw}

\title{Prism graphs in tropical plane curves}
\author{Liza Jacoby, Ralph Morrison, Ben Weber}
\date{}

\usepackage{graphicx}

\begin{document}

\maketitle

\centerline{\textbf{Abstract}}

Any smooth tropical plane curve contains a distinguished trivalent graph called its skeleton.  In 2020 Morrison and Tewari proved that the so-called big face graphs cannot be the skeleta of tropical curves for genus $12$ and greater. In this paper we answer an open question they posed to extend their result to the prism graphs, proving that they are the skeleton of a smooth tropical plane curve precisely when the genus is at most $11$.  Our main tool is a classification of lattice polygons with two points than can simultaneously view all others, without having any one point that can observe all others.

\section{Introduction}

A tropical plane curve is a combinatorial analog of an algebraic plane curve, and can be defined using a polynomial over the min-plus semiring \cite{ms}.  Alternatively, it can be defined as a polyhedral complex dual to a regular subdivision of a convex lattice polygon, i.e., one with integer coordinates.  Under this duality, vertices of a tropical curve $\Gamma$ correspond to $2$-dimensional cells of a subdivision $\mathcal{T}$ of a lattice polygon $P$, and two vertices of $\Gamma$ are connected by an edge if and only if the corresponding $2$-cells share an edge; the edges $\Gamma$ are then perpendicular to the edges of $\mathcal{T}$.  Similarly, rays of $\Gamma$ are dual to the subdivided boundary edges of $P$, again perpendicular.
We call $\Gamma$ \emph{smooth} if  $\mathcal{T}$ is a unimodular triangulation, meaning that every $2$-dimensional cell is a triangle of area $\frac{1}{2}$.   An example of a unimodular triangulation of a lattice polygon appears on the left in Figure \ref{figure:prism_g4_curve}, with a dual smooth tropical plane curve in the middle.

\begin{figure}[hbt]
    \centering
    \includegraphics[scale=1]{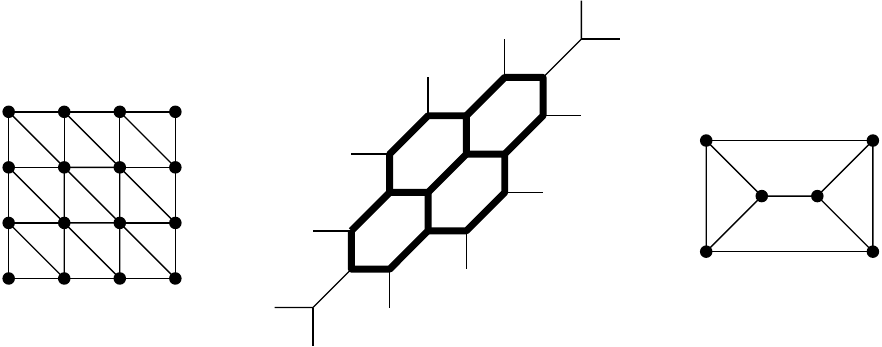}
    \caption{A unimodular triangulation of a lattice polygon; a dual tropical curve; and its skeleton}
    \label{figure:prism_g4_curve}
\end{figure}

The \emph{genus} of the polygon $P$ is defined to be the number of interior lattice points.  Note that if $\Gamma$ is a smooth tropical plane curve dual to a unimodular triangulation of $P$, the number of bounded faces of $\Gamma$ is equal to $g$; we also call this number the genus of $\Gamma$.  A smooth plane tropical curve of genus $g\geq 2$ contains a distinguished graph called its \emph{skeleton}.  This is a trivalent graph that is obtained from $\Gamma$ by removing all rays and iteratively contracting all $1$-valent vertices, and then smoothing over any $2$-valent vertices.  In the tropical curve in Figure \ref{figure:prism_g4_curve}, the subset of the curve contributing to the skeleton is highlighted; the skeleton itself is illustrated on the right.   Although the skeleton is usually viewed as an abstract graph, it {inherits} a planar embedding from the tropical curve, since a subdivided copy of the skeleton is embedded in the plane as a subset of the tropical curve.

Sometimes the skeleton is considered as a \emph{metric} graph, meaning that there are lengths assigned to the edges based on the embedding of the tropical curve \cite{bjms,small2017dimensions}.  In this paper, however, we consider the skeleton as a purely combinatorial graph. Following \cite{small2017graphs}, we say that a graph is \emph{tropically planar} if it is the skeleton of some smooth tropical plane curve.  There are a number of necessary conditions for a graph to be tropically planar; for instance, it must be connected, planar, and trivalent.  There exist many results in the literature that provide additional constraints on tropically planar graphs, often by describing forbidden patterns that cannot appear in such graphs \cite{cdmy,small2017graphs,forbidden_patterns,morrison-hyperelliptic}. 

One recent contribution to this family of results came in \cite{panoptigons}, which studied the so-called \emph{big-face graphs} through the lens of tropical planarity.  A planar graph $G$ is called a big-face graph if for every planar embedding of $G$, at least one bounded face (the ``big face'') shares an edge with every other bounded face.  By \cite[Theorem 1.2 and Appendix A]{bjms}, there do not exist any tropically planar big-face graphs of genus $g\geq 12$.  Their argument goes as follows:  if a big-face graph is tropically planar, then in a dual triangulation there must be a lattice point (dual to the big face) connected to every other interior lattice point.  This means that the interior polygon $P_\textrm{int}$, defined to be the convex hull of the interior lattice points, has all lattice points \emph{visible} (see Definition \ref{defn:visibility}) from a single one of its lattice points; they call such a polygon a \emph{panoptigon}.  By proving that a panoptigon cannot be the interior polygon of a lattice polygon with genus $g\geq 12$, they arrive at the desired result.

In this paper we study a related family of graphs through the lens of tropically planar graphs, namely the \emph{prism graphs}.  For $n\geq 3$, the prism graph $\mathfrak{p}_n$ of genus $n+1$ is  the Cartesian product  $C_n\square K_2$ of a cycle of length $n$ with a path of length $2$.  This is intuitively a pair of cycles, each on $n$ vertices, with each pair of matching vertices connected by an edge.  The usual depiction of a prism graph is as one $n$-cycle inside of another, with $n$ edges connecting the two; in this embedding, there is a bounded face bordering all other faces.  However, most prism graphs are not big-face graphs: they have an alternate embedding that places a $4$-cycle as the unbounded face instead of an $n$-cycle.  The two embeddings of the prism graph of genus $g$ are pictured in Figure \ref{figure:low_genus_prisms} for $4\leq g\leq 7$.  We remark that for $g=4$ and $g=5$, both embeddings have a big face (in fact, the two embeddings of the genus-5 prism $\mathfrak{p}_4$ are combinatorially isomorphic), so the prism graph $\mathfrak{p}_n$ is a big-face graph if and only if $n\leq 4$.  For $n\neq 4 $, we call the big-face embedding the \emph{standard embedding} of $\mathfrak{p}_n$, and the other embedding the \emph{non-standard embedding}.  (Note that since the prism graph is $3$-connected, it has a unique embedding on the sphere by \cite{whitney}.  This means any planar embedding is determined by choosing the unbounded face; thus there really are only the two distinct embeddings for most prisms.)

\begin{figure}[h!]
    \centering
    \includegraphics[scale=0.8]{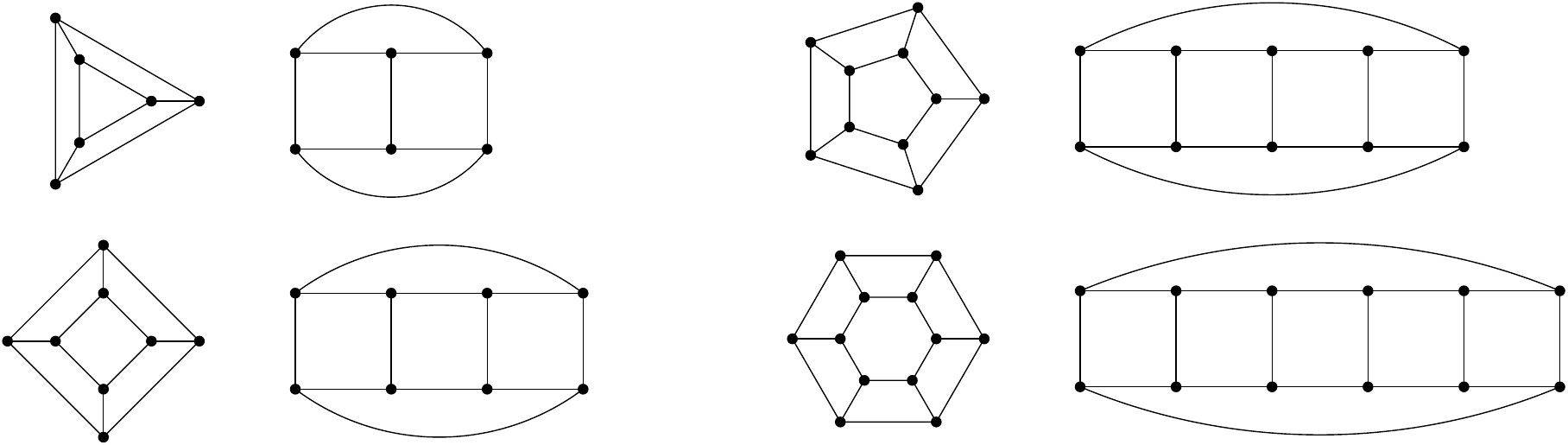}
    \caption{The standard and non-standard embeddings for small prisms (equivalent for $g=5$)}
    \label{figure:low_genus_prisms}
\end{figure}

Our main result is the following, which answers an open question posed in \cite{panoptigons}.

\begin{theorem}\label{thm:prismgraph}
The prism graph of genus $g\geq 4$ is tropically planar if and only if $g \leq 11$.  Moreover, for \(5\leq g\leq 11\) only the standard embedding of the prism graph can be inherited from a smooth tropical plane curve.
\end{theorem}

The ``if'' direction of this theorem is handled by examples of tropical curves of genus $g$ from $4$ to $11$ whose skeleta are prism graphs. We already saw the genus $4$ prism in Figure \ref{figure:prism_g4_curve}, where it inherited the nonstandard embedding from a tropical curve. Examples for \(4\leq g\leq 11\) with the standard embedding of the prism graph appear in Figure \ref{figure:prism_curves_5_to_11}.
For the ``only if'' direction, we must consider the two embeddings of the prism graph:  the standard and the non-standard embedding.  The standard embedding has a big-face, and so could only arise from a polygon $P$ whose interior polygon is a panoptigon; by \cite{panoptigons} it follows that such a polygon $P$ (and thus the prism) has genus at most $11$.

\begin{figure}[h!]
    \centering
    \includegraphics[scale=1]{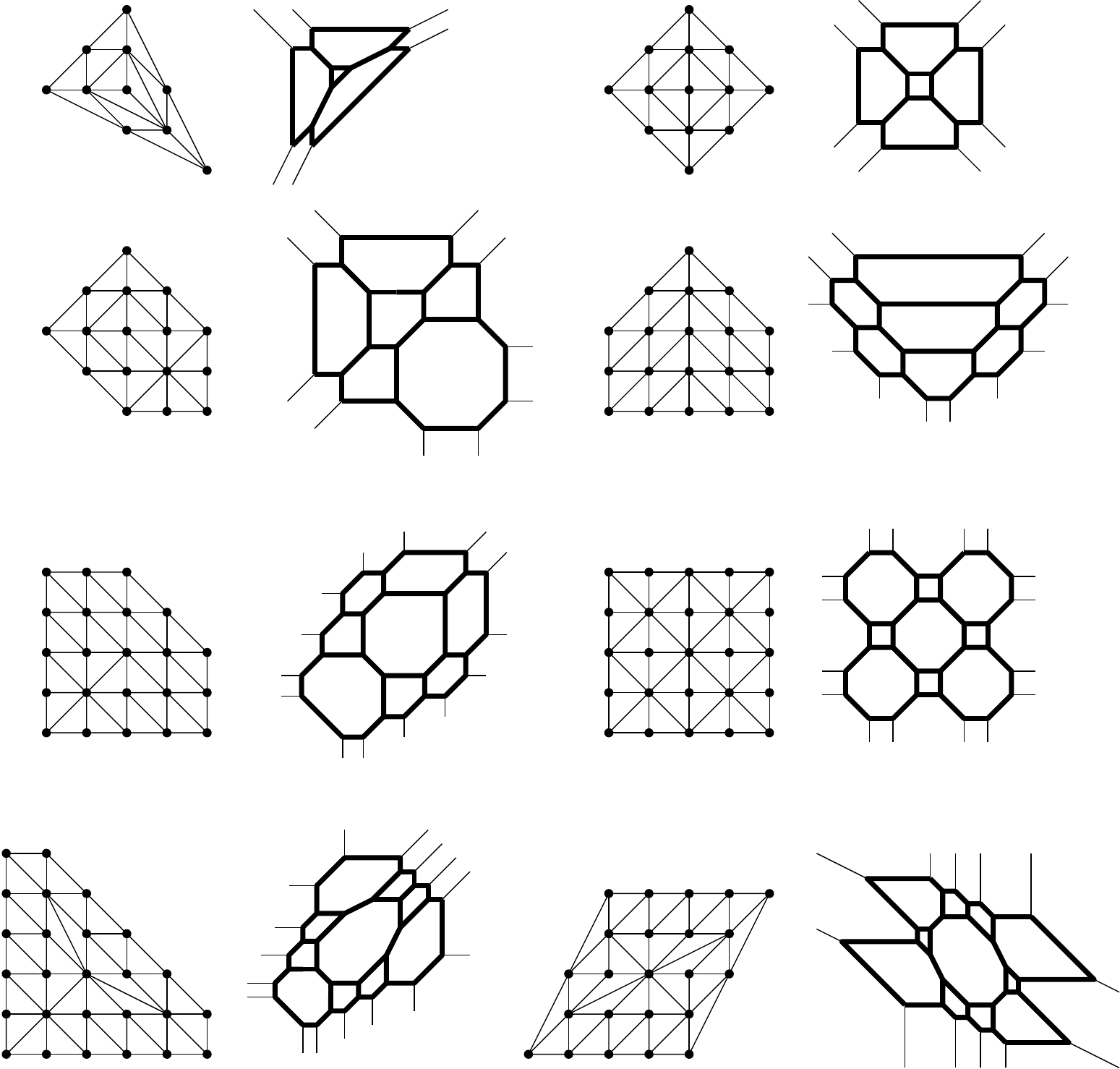}
    \caption{Tropical curves whose skeleta are prisms}
    \label{figure:prism_curves_5_to_11}
\end{figure}

It remains to consider the non-standard embedding, which has two bounded faces $F_1$ and $F_2$ that do not share an edge with one another, although each shares an edge with every other bounded faces. 
Dually, this means that a Newton polygon $P$ has two interior lattice points $p_1$ and $p_2$ with a triangulation that does not connect $p_1$ and $p_2$, but does connect both $p_1$ and $p_2$ to every other interior lattice point.  On the level of the interior polygon $P_\textrm{int}$, this means that either $P_\textrm{int}$ is a panoptigon, or $P_\textrm{int}$ has two lattice points that are not visible to one another, but that are visible to all other lattice points, such that none of these lines of sight cross.  In the latter case we call $P_\textrm{int}$ a \emph{diploptigon}.  By classifying these polygons, we will show that they cannot in fact ever be the interior polygon of a lattice polygon, yielding our desired  ``if and only if'' result.  The reason that we cannot have the nonstandard embedding for \(g\geq 5\) is discussed in Remark \ref{remark:embeddings}.


Our paper is organized as follows. In the Section \ref{sec:latticepolygons}, we provide the necessary background on lattice polygons, then in Section \ref{section:small_lattice_diameter} classify all diploptigons (indeed, all lattice polygons) of lattice diameter at most $2$. In Section \ref{section:high_lattice_diameter}, we classify the geometry of all diploptigons with lattice width greater than $3$. We then prove Theorem \ref{thm:prismgraph} using our complete diploptigon classification.

\bigskip

\noindent{\textbf{Acknowledgments.}}  The authors would like to thank Marino Echavarria, Max Everett, Shinyu Huang, Ayush Tewari, and Raluca Vlad for maybe helpful discussions on tropical curves and their skeleta.  The authors were supported by NSF grant DMS-1659037 and by Williams College.

\section{Lattice polygons}\label{sec:latticepolygons}

In this section, we outline important definitions and results pertaining to lattice visibility lattice polygons, including the notions of lattice width, relaxed polygons, and panoptigon points. 
Throughout the remainder of the paper, we will take \emph{polygon} to refer to a two-dimensional convex lattice polygon.

A \emph{unimodular transformation} is a map  $t:\mathbb{R}^2\rightarrow\mathbb{R}^2$ of the form $t(p)=Ap+b$, where $A=\left(\begin{smallmatrix}a&b\\c&d\end{smallmatrix}\right)$ is an integer matrix of determinant $\pm1$ and $b$ is a translation vector.  We say that a polygon $P'$ is \emph{equivalent} to a polygon $P$ if there exists a unimodular transformation $t$ with $t(P)=P'$.  One useful set of transformations are the \emph{shearing} transformations, which have $A$ of the form $\left(\begin{smallmatrix}1&N\\0&1\end{smallmatrix}\right)$ and $\left(\begin{smallmatrix}1&0\\N&1\end{smallmatrix}\right)$.

We say that a polygon $P$ is \emph{hyperelliptic} if its interior points are collinear, and \emph{non-hyperelliptic} if its interior points are not all collinear.  The convex hull of the interior lattice points of $P$ is called the \emph{interior polygon} of $P$, denoted $P_{\Int}$.  Thus a polygon is non-hyperelliptic if and only if $P_{\Int}$ is $2$-dimensional.  If $P$ is a polygon of genus $g$ that is contained in no other lattice polygon of genus $g$, we say that $P$ is a \emph{maximal} polygon.

Any lattice polygon $P$ can be defined as the intersection of finitely many half-planes $P=\bigcap_{i=1}^n \mathcal{H}_{i}$, where the half-plane $\mathcal{H}_i$ is the set of all points $(x,y)$ satisfying $a_ix+b_iy\leq c_i$, with $a_i,b_i,c_i$ relatively prime integers.  For each $\mathcal{H}_i$, define the corresponding \emph{relaxed half-plane $\mathcal{H}_i^{(-1)}$} to the the half-plane determined by the inequality $a_ix+b_iy\leq c_i+1$.  Then, the \emph{relaxed polygon $P^{(-1)}$ of $P$} is the intersection of the relaxed half-planes:
\[P^{(-1)}=\bigcap_{i=1}^n \mathcal{H}_{i}^{(-1)}. \]
The relaxation process, along with the corresponding relaxed polygons, is illustrated in Figure \ref{figure:relaxation_process} for two polygons with five lattice points.  We note that the relaxed polygon need not be a lattice polygon:  in the second case, the relaxation process introduces a vertex that does not have integer coordinates.  By the construction of $P^{(-1)}$, the lattice points interior to $P^{(-1)}$ are precisely the lattice points of $P$.

\begin{figure}[hbt]
    \centering
    \includegraphics[scale=1]{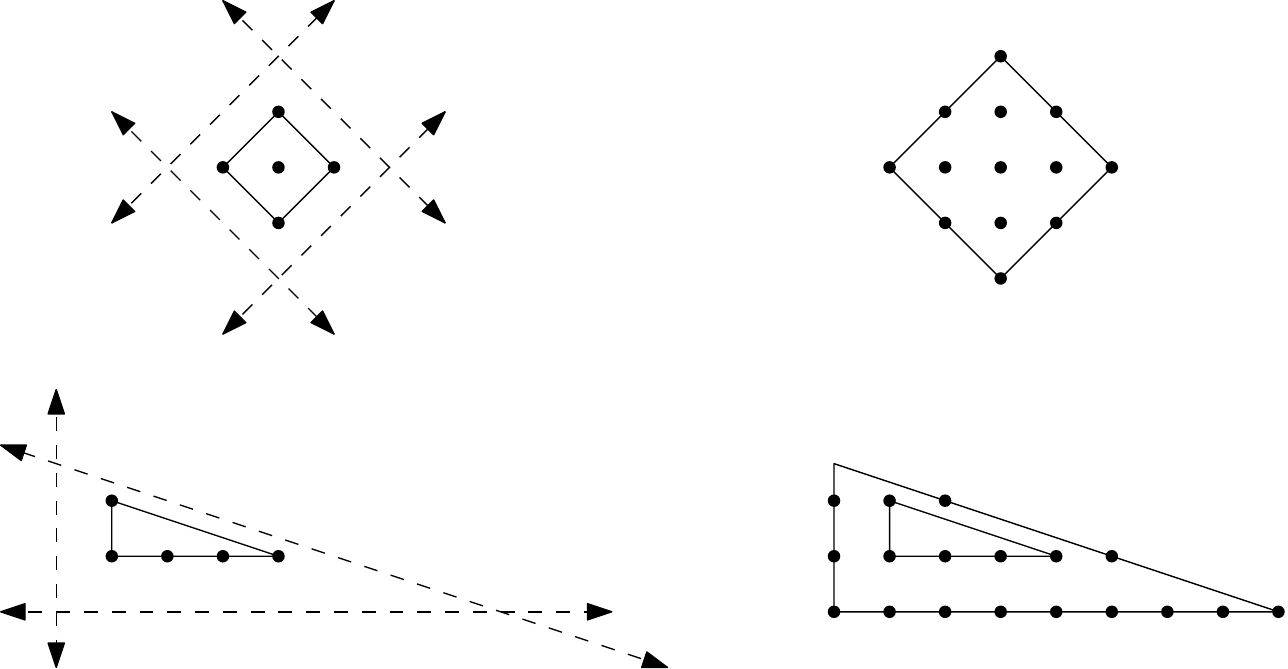}
    \caption{The relaxation process for two lattice polygons, each with five lattice points (left), along with the corresponding relaxed polygon (right).  The top relaxed polygon is a lattice polygon; the bottom is not.}
    \label{figure:relaxation_process}
\end{figure}

By \cite[Lemma 2.2.13]{Koelman}, if $P$ is a non-hyperelliptic lattice polygon, then $(P_{\textrm{int}})^{(-1)}$ is a lattice polygon containing $P$; in fact, it is the unique maximal lattice polygon containing $P$ that has the same interior polygon.  This implies the following lemma.

\begin{lemma}\label{lemma:not_lattice}
Let $Q$ be a lattice polygon.  Then $Q^{(-1)}$ is a lattice polygon if and only if $Q=P_\textrm{int}$ for some lattice polygon $P$.
\end{lemma}

It is worth remarking that \(Q^{(-1)}\) need not equal \(P\).  For instance, if \(Q\) is an isosceles right triangle with base and height \(1\), than \(Q^{(-1)}\) is such a triangle with base and height \(4\); but there are many other polygons \(P\) with \(P_\textrm{int}=Q\), such as all those pictured in Figure \ref{figure:trapezoid_argument}.


We now recall terminology relevant to lattice point visibility, including the following key definition.

\begin{definition}\label{defn:visibility}  We say two lattice points $p$ and $q$ are \emph{visible} to one another if the line segment $\overline{pq}$ contains no other lattice points besides $p$ and $q$.
\end{definition}

 Following \cite{panoptigons}, a \emph{panoptigon} is a polygon $P$ such that every lattice point $q \in P\cap\mathbb{Z}^2$ is visible from some fixed $p \in P\cap\mathbb{Z}^2$.  We say that $P$ is a \emph{diploptigon} if it is not a panoptigon, but it contains lattice points $p$ and $q$ such that all lattice points in $P$ besides $q$ are visible to $p$; all besides $p$ are visible to $q$; and that none of these lines of sight cross.  In other words, if we draw all line segments from $p$ to lattice points besides $q$, and from $q$ to lattice points besides $p$, no two of those line segments cross, and none contain any lattice points besides their end points.  If $P$ is a diploptigon, then any pair $\{p,q\}$ satisfying these conditions is called a \emph{pair of diploptigon points}; when the pairing of $p$ and $q$ is clear, we simply refer to them as diploptigon points.  Several diploptigons appear in Figure \ref{figure:diplo_examples}, each with a pair of diploptigon points circled; note that this pair need not be unique, and also that lines of sight can lie along the edges of the polygon.

\begin{figure}[hbt]
    \centering
    \includegraphics[scale=1]{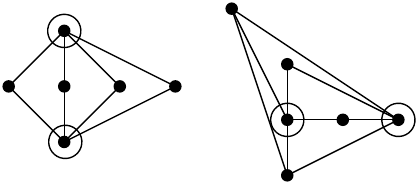}
    \caption{Two diploptigons, each with a pair of diploptigon points circled and lines of sight drawn}
    \label{figure:diplo_examples}
\end{figure}

The following result will be useful in proving that certain polygons are not diploptigons.

\begin{lemma}\label{lemma:diplo_lines}
If $P$ is a diploptigon with diploptigon points \(p\) and \(q\), then the line $L$ passing through \(p\) and \(q\) satisfies $|L\cap P\cap\mathbb{Z}^2|=3$, and for every other line $L'$ parallel to $L$ we have $|L'\cap P\cap\mathbb{Z}^2|\leq 1$.
\end{lemma}

\begin{proof}
First we show that  $|L\cap P\cap\mathbb{Z}^2|=3$.  Certainly $L$ contains the two points $p$ and $q$; if it contains no other lattice points in $P$, then we would have that $p$ and $q$ are panoptigon points, implying that $P$ is a panoptigon rather than a diploptigon.  Thus $|L\cap P\cap\mathbb{Z}^2|\geq 3$.  If $|L\cap P\cap\mathbb{Z}^2|\geq 4$, then at least one lattice point outside of $\{p,q\}$ would be invisible to one of $p$ or $q$, contradicting our visibility assumptions.  Thus $|L\cap P\cap\mathbb{Z}^2|= 3$.

Let $L'$ be distinct from and parallel to $L$, and suppose for the sake of contradiction that $L'\cap P\cap\mathbb{Z}^2$ contains two points $r$ and $s$.  Since $L'\neq L$, we have that $p,q,r,$ and $s$ form a trapezoid.  Among the four lines of sight from $\{p,q\}$ to $\{r,s\}$, a pair of the lines of sight form the diagonals of the parallelogram.  These cross, contradicting the assumption that $p$ and $q$ are diploptigon points.  Thus $|L'\cap P\cap\mathbb{Z}^2|\leq 1$, as claimed.
\end{proof}

A helpful set of tools in classifying diploptigons will be
the lattice diameter of a polygon $P$, denoted by $\ld(P)$.  This is the largest $d$ such that there exist $d+1$ collinear points in $P\cap\mathbb{Z}^2$. 
A somewhat dual notion is the \emph{lattice width} of $P$, written $\lw(P)$. This is the smallest $w$ such that there exists a polygon $P'$ equivalent by a unimodular transformation to $P$ contained in the horizontal strip $\mathbb{R}\times [0,w]$.

We close this section by recalling two useful properties of lattice diameter and lattice width.

\begin{lemma}[Equation (2) and Theorem 3 in \cite{lattice_diameter}]\label{lemma:ld_facts}
Let $P$ be a lattice polygon. Then
\begin{itemize}
    \item[(a)] $\lw(P) \leq \left\lfloor \frac{4}{3} \ld(P) \right\rfloor + 1$, and
    \item[(b)] $| P \cap \Z^2 | \leq ( \ld(P) + 1)^2$.
\end{itemize}
\end{lemma}

\section{Polygons with small lattice diameter}
\label{section:small_lattice_diameter}

In this section we will classify all polygons with lattice diameter at most $2$, up to equivalence.  We will then pick out those that are diploptigons.  

For the case of $\ld(P)=1$, Lemma \ref{lemma:ld_facts}(b) implies that $ | P \cap \Z^2 |\leq 4$.  Any polygon has at least $3$ lattice boundary points, so such a $P$ must have genus $g=0$ or $g=1$.  Consulting a classification of all polygons of genus at most $1$ as in \cite{movingout}, we see that there are precisely $4$ such polygons, $3$ of which have lattice diameter equal to $1$.  These are pictured in Figure \ref{figure:4_lattice_points}.

\begin{figure}[hbt]
    \centering
    \includegraphics[scale=1]{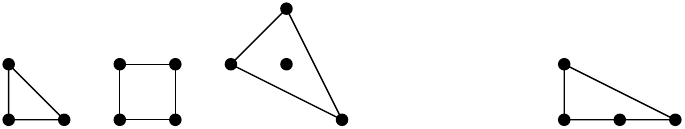}
    \caption{All polygons with at most $4$ lattice points, up to equivalence; the three on the left have lattice diameter $1$}
    \label{figure:4_lattice_points}
\end{figure}

We now move on to those polygons $P$ with $\ld(P)=2$, which have $|P\cap\mathbb{Z}^2|\leq (2+1)^3=9$.  By Lemma \ref{lemma:ld_facts}(a) we know that these polygons have $\lw(P)\leq\left\lfloor\frac{4}{3}\cdot 2\right\rfloor+1=3$.  First assume $\lw(P)\leq 2$. All lattice polygons of lattice width at most $2$ are classified in \cite[Theorem 2.3]{panoptigons}, a result modeled off a similar classifications found in \cite{movingout} and  \cite{Koelman}.  Polygons of lattice width $1$ and lattice diameter at most $2$ are (possibly degenerate) trapezoids of height $1$ with at most $3$ lattice points on each horizontal edge; and those of lattice width $2$ are hyperelliptic polygons of genus  at most $2$, since those of higher genus have at least $4$ lattice points at some common height.  Those that end up having lattice diameter equal to $2$ are pictured in Figure \ref{figure:lattice_width_leq2}.

\begin{figure}[hbt]
    \centering
    \includegraphics[scale=0.7]{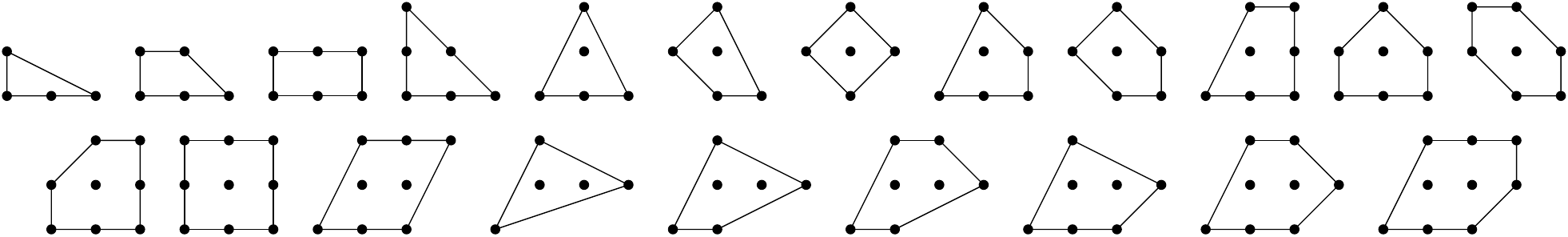}
    \caption{Polygons of lattice diameter $2$ and lattice width at most $2$}
    \label{figure:lattice_width_leq2}
\end{figure}

Now we find those polygons with lattice diameter $2$ and lattice width $3$.  As argued in the proof of  \cite[Proposition A.1]{panoptigons}, such a polygon $P$ must have one of the first two polygons in Figure \ref{figure:4_lattice_points} as $P_{\Int}$, and due to its lattice width must be a subpolygon of one of the polygons appearing in Figure \ref{figure:candidate_superpolygons}.  For each of these two polygons, we run through a series of arguments to find all lattice diameter-$2$ subpolygons, up to equivalence; these are diagrammed in Figures \ref{figure:trapezoid_argument} and \ref{figure:square_argument}.  The key argument in each case is that we cannot have $4$ collinear lattice points

\begin{figure}[hbt]
    \centering
    \includegraphics[scale=1]{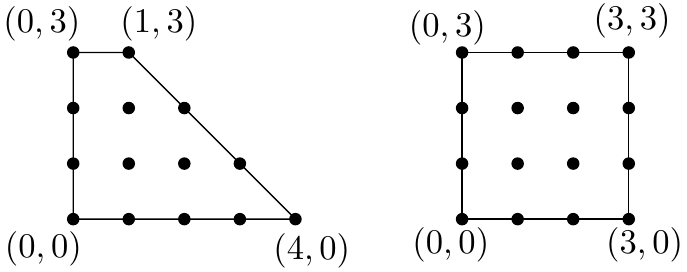}
    \caption{Candidate superpolygons of a polygon of lattice diameter $2$ and lattice width $3$}
    \label{figure:candidate_superpolygons}
\end{figure}

For the trapezoid, whose argument is diagrammed in in Figure \ref{figure:trapezoid_argument}, we know that we must omit at least one of the points $(0,1)$ and $(3,1)$, since there cannot be $4$ points at height $1$. These points are symmetric, so without loss of generality we will omit $(0,1)$.  This means we either have to omit the points $(0,2)$ and $(0,3)$, or the point $(0,0)$.  In the first case, we can no longer omit $(1,3)$, since we must keep at least one point with \(y\)-coordinate \(3\) to preserve the interior polygon; so the only way to avoid having $4$ lattice points along the diagonal edge is to omit $(4,0)$.  This yields the triangle with vertices at $(0,0)$, $(3,1)$ and $(3,1)$; this does indeed have lattice diameter $2$, and contains no proper subpolygons with the same interior polygon.  In the second case, with $(0,0)$ omitted, we must omit at least one more point at height $0$.  It cannot be $(1,0)$, and so $(4,0)$ must be omitted.  Similarly, some point with $x$-coordinate $1$ must be omitted, and it cannot be $(1,0)$, so it must be $(1,3)$.  Finally, at least one of the points $(0,3)$, $(1,2)$, $(2,1)$, and $(3,0)$ must be omitted.  It cannot be $(0,3)$, so it must be $(3,0)$.  This yields finally a polygon with lattice diameter $2$.  We can still take subpolygons by removing some subset of $\{(0,2),(2,0),(2,2)\}$; up to symmetry it only matters how many are removed, yielding $3$ more polygons of lattice diameter $2$.  

\begin{figure}[hbt]
    \centering
    \includegraphics[scale=1]{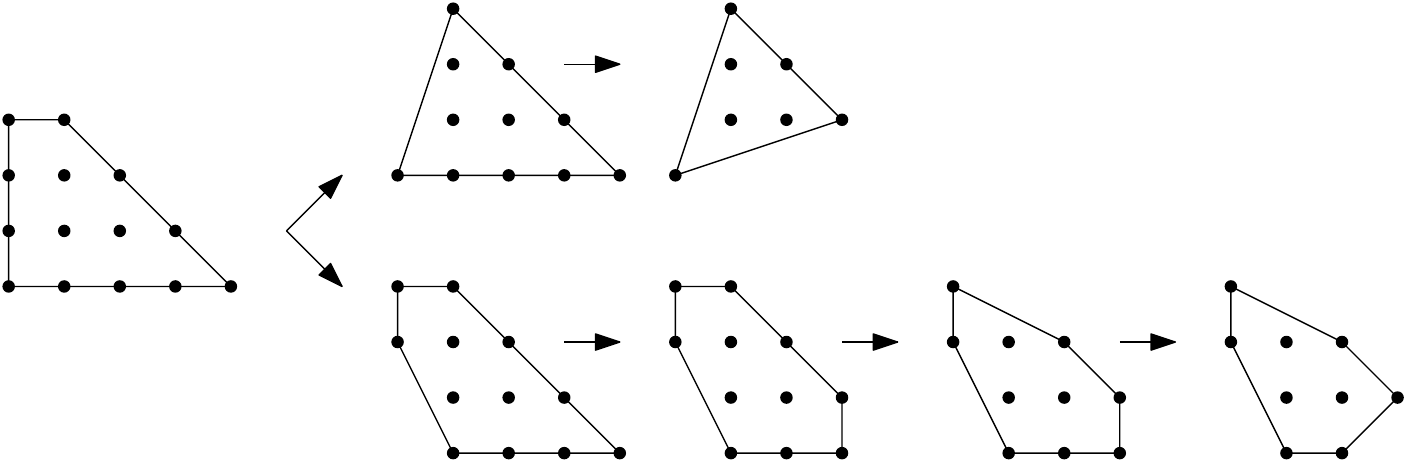}
    \caption{The argument for subpolygons of the trapezoid}
    \label{figure:trapezoid_argument}
\end{figure}

This argument for the square is diagrammed in Figure \ref{figure:square_argument}.  We claim that three vertices must be omitted.  Certainly one must be omitted from height $0$; without loss of generality, say it is $(0,0)$.  From there one must be omitted from the still-complete diagonal; without loss of generality, say it is $(0,3)$.  Finally, one must be omitted from the rightmost column; without loss of generality, say it is $(3,3)$.  Since there are $4$ lattice points with $x$-coordinate $1$, either $(1,0)$ or $(1,3)$ must be omitted. If it is $(1,0)$, then we split into two subcases: one where we remove $(2,0)$, and one where we remove $(2,3)$.  If it is $(2,0)$, then the only way to avoid having $4$ points at heights $1$ and $2$ is to omit $(3,1)$ and $(3,2)$; this yields a pentagon that has lattice diameter $2$, with vertices at $(0,1)$, $(0,2)$, $(1,3)$, $(2,3)$, and $(3,0)$.  We can remove exactly one of $(0,2)$ and $(1,3)$ to obtain a subpolygon with lattice diameter $2$; these polygons are equivalent to one another, and so contribute one more.  If instead of $(2,0)$ we removed $(2,3)$, we must remove $(0,2)$ and $(3,1)$ to avoid $4$ collinear points, yielding a square that has lattice diameter $2$.  Moving back, if instead of $(1,0)$ we instead remove $(1,3)$, then the only way to avoid having $4$ lattice points with $x$-coordinate $2$ is to omit $(1,0)$ and $(2,0)$; this pushes us back into the prior case pictured at the top of Figure \ref{figure:square_argument}, and so won't find any lattice diameter $2$ polygons we have not already found.

\begin{figure}[hbt]
    \centering
    \includegraphics[scale=1]{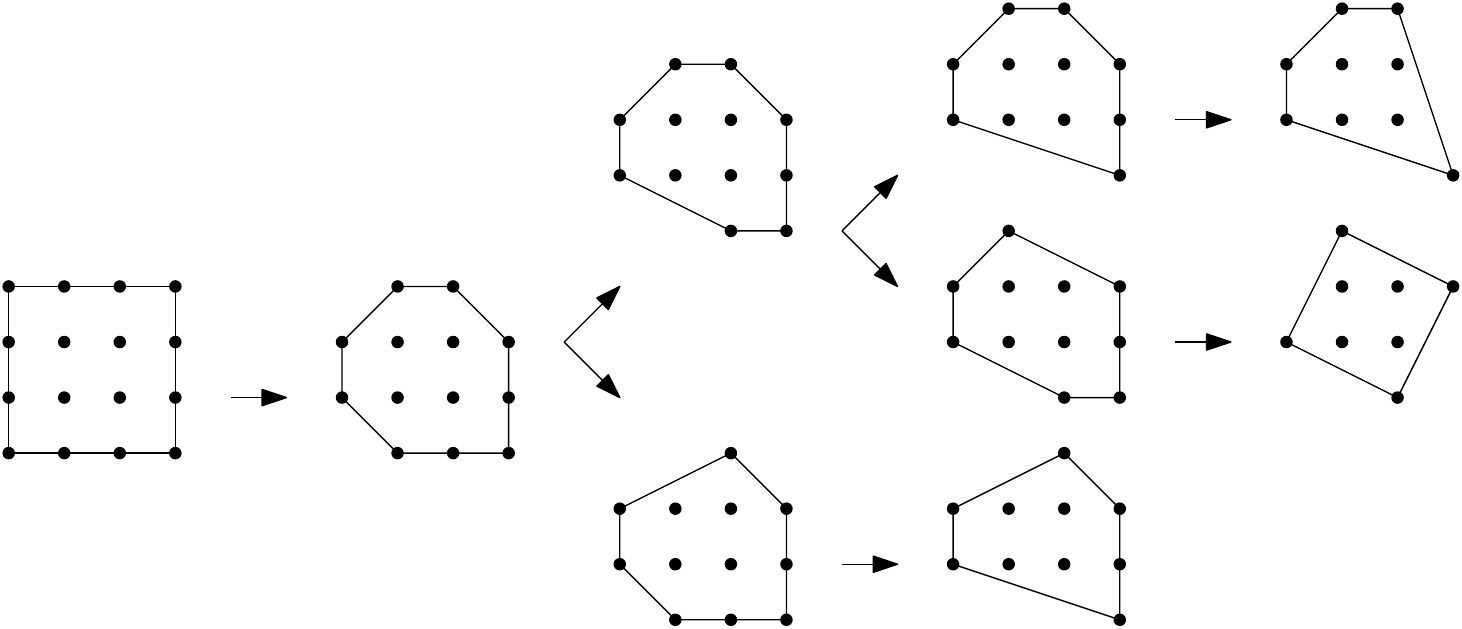}
    \caption{The argument for subpolygons of the square}
    \label{figure:square_argument}
\end{figure}

We thus have precisely eight polygons with lattice diameter $2$
 and lattice width $3$, namely those illustrated in Figure \ref{figure:polygons_ld2lw3}.
 
 \begin{figure}[hbt]
    \centering
    \includegraphics[scale=0.8]{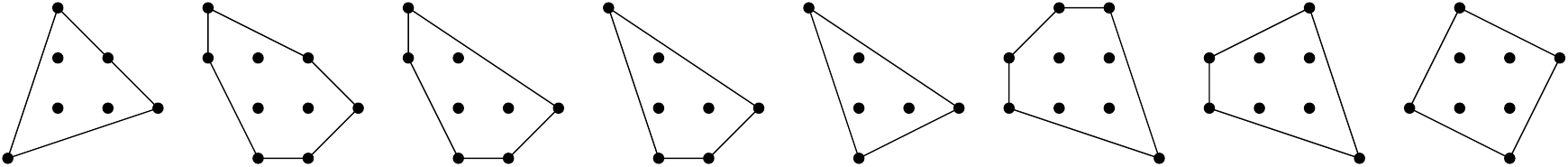}
    \caption{The polygons of lattice diameter $2$ and lattice width $3$}
    \label{figure:polygons_ld2lw3}
\end{figure}

Having found all $32$ polygons of lattice diameter at most $2$, we determine which are diploptigons.  All those from Figures \ref{figure:4_lattice_points} and \ref{figure:lattice_width_leq2} are panoptigons, and so are omitted from our consideration.  Of those polygons in Figure \ref{figure:polygons_ld2lw3}, we see that the first is a panoptigon, and the fifth is a diploptigon; for instance, we may choose $p$ and $q$ to be the points $(1,1)$ and $(3,1)$.  For the other six polygons, one can check that for every line passing through three lattice points, there exists a parallel line passing through two or more lattice points.  By Lemma \ref{lemma:diplo_lines}, this implies that none of them is a diploptigon.  We summarize this in the following proposition.

\begin{proposition}\label{prop:diplo_ld2}
Up to equivalence, there is exactly one diploptigon of lattice diameter at most $2$, namely the triangle with vertices at $(1,0)$, $(3,1)$, and $(0,3)$.
\end{proposition}

\section{Diploptigons with lattice diameter at least $3$}
\label{section:high_lattice_diameter}

We now classify all diploptigons of lattice diameter $3$ or more. Our main strategies mirror the proof of \cite[Theorem 1.1]{panoptigons}. Let $P$ be such a diploptigon, containing points $p$ and $q$ from which every lattice point is visible except for $q$ and $p$, respectively.  After translating, we may assume that $p$ is the origin $(0,0)$.  Since $\ld(P)\geq 3$, we know that there exist $4$ collinear lattice points in $P$.  After applying a unimodular transformation that fixes the origin, we may assume that they lie on a horizontal line. Thus $P$ contains points of the form $(a+n, b)$, $n \in \{0,1,2,3\}$, all of which are visible from the origin; note that \(q\) cannot be visible to \(p\), and so is not among these points. We know that $b \neq 0$, as the only points visible from $p$ at height $0$ are itself, $(1,0)$, and $(-1,0)$. 
We also cannot have $b$ even:  at an even nonzero height, every other point (namely those with an even $x$-coordinate) is invisible from the origin, so it is impossible to have $4$ visible points in a row.  Thus $b$ must be odd.

By performing a rotation if necessary, we may assume that $b\leq -1$.  We claim that $b=-1$.  For the sake of contradiction, suppose not, so that $b\leq -3$.  Consider the triangle $T$ with vertices at $(0,0)$, $(a,b)$, and $(a+3,b)$.  By convexity, $T\subset P$.  Since $P$ is a diploptigon with $p=(0,0)$ as one of its diploptigon points, we know that every lattice point in $T$ is either visible, or is the other diploptigon point $q$. Let $L$ be the line defined by $y=b+1$, and consider the line segment $L\cap T$.  Since the cross-sectional width of $T$ is equal to $3$ at height $b$ and $0$ and height $0$ and decreases linearly, we have that the width of $L\cap T $ is $3-\frac{1}{|b|}\geq 3-\frac{1}{3}>2$.  Since this horizontal line segment has width greater than $2$ and has integral height, it must contain at least $2$ consecutive lattice points.  One of them must have an even $x$-coordinate, and so be invisible to the origin.  This point is contained in $T$, and thus in $P$; because it is invisible to $p=(0,0)$, it must be $q$.  This means that the line segments connecting $q$ to $(a,b)$, $(a,b+1)$, $(a,b+2)$, and $(a,b+3)$ must not cross any of the line segments connecting $(0,0)$ to the same set of points.  But this is impossible: the line segment connecting $(0,0)$ to $(a,b+1)$ must cross one of the line segments connecting $q$ to $(a,b)$ or $(a,b+3)$, yielding a contradiction.  Thus we have that $b=-1$.

We now have our diploptigon $P$ with $p$ at the origin and the points $a_1$, $a_2$, $a_3$, and $a_4$ at $y=-1$.  Consider the other diploptigon point $q=(c,d)$.  In order for its lines of sight not to cross those of $p$, we need $d\leq -2$.  Suppose $d<-2$. A symmetric argument to the above (with $p$ and $q$ swapped) shows that $d$ must be even, and that if $d\leq -4$ the convex hull of $q$ with $a_1,\ldots,a_4$ would introduce a point invisible to $q$, with negative $y$-coordinate; since $p$ is the only point not visible to $q$, this is impossible, so $d=-2$.  As $q$ is not visible to $p$, we have $q=(2N,-2)$ for some $N$,  Applying the shearing transformation $\begin{pmatrix} 1 & N \\ 0 & 1 \end{pmatrix}$, we have that $q=(0,-2)$.

We now deal with three cases, based on where lattice points other than $p$ and $q$ can be found.  The first two cases give us infinite families of diploptigons, as well as an exceptional diploptigon; the third case gives us none. The resulting diploptigons are illustrated in Figure \ref{figure:diplo_types_ld3}, along with the lines of sight from the diploptigon points $(0,0)$ and $(0,-2)$.

\begin{figure}[hbt]
    \centering
    \includegraphics[scale=1]{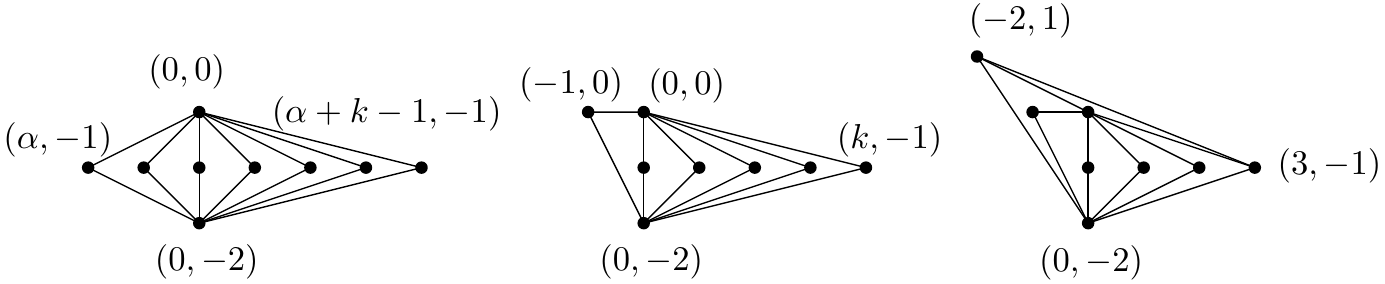}
    \caption{The three types of diploptigons of lattice diameter at least $3$}
    \label{figure:diplo_types_ld3}
\end{figure}

\begin{itemize}
    \item \textbf{Case 1: additional points are only at height $-1$.}
 In this case we have the (possibly degenerate) quadrilaterals with vertices at $(0,0)$, $(0,-2)$, $(x_1,-1)$, and $(x_2,-1)$, where $x_1\leq 0\leq x_2$ and $x_2-x_1\geq 3$.
    \item \textbf{Case 2: additional points appear on at least one height among $0$ and $-2$.}  
Let $a_1=(\alpha,-1),\ldots,a_k=(\alpha+k-1,-1)$ be the points at height $-1$.  By assumption, $k\geq 4$, and by convexity $\alpha\leq 0\leq \alpha+k-1$.  Let us now consider which other points at heights $0$ and $-2$ could be included in $P$.  Certainly they would have to be visible to both $p$ and $q$, and so would have to have $x$-coordinate $\pm 1$.  Without loss of generality, if such a point is included, we may assume that it is $(-1,0)$.  We then have that $\alpha=0$: otherwise there would be two lattice points on a vertical line, contradicting Lemma \ref{lemma:diplo_lines}. The same argument shows that we may not include $(-1,-2)$; we can also rule out $(1,0)$ and $(1,-2)$ by lines of sight.

We then ask whether $P$ could contain any other lattice point $(e,f)$.  We could not have $f\leq -2$:  such a point would either create a conflict with a line of sight, or would introduce points invisible to $q$ at height $-2$.  Thus we need $f\geq 1$.  We also need the point $(e,f)$ to lie strictly above the line $y=-x-2$: otherwise we would introduce the point $(-1,-1)$ by convexity since we already have $(0,-2)$ in $P$, violating Lemma \ref{lemma:diplo_lines}.  Similarly, we need $(e,f)$ to lie strictly below the line $y=-\frac{x}{2}+\frac{1}{2}$: otherwise we would introduce the point $(1,0)$, again violating Lemma \ref{lemma:diplo_lines}.  The only point satisfying all these criteria is $(-2,1)$, as illustrated in Figure \ref{figure:diplo_lines}.  However, this point can only be included if $k=4$: otherwise the presence of the point $(4,-1)$ will introduce $(1,1)$.  Thus we can add an additional point if and only if $k=4$, and this must be the point $(-1,1)$.  So, we find an infinite family of diploptigons when there are no points outside of the strip $-2\leq y\leq 0$, and a single diploptigon when there are.

\begin{figure}[hbt]
    \centering
    \includegraphics[scale=1]{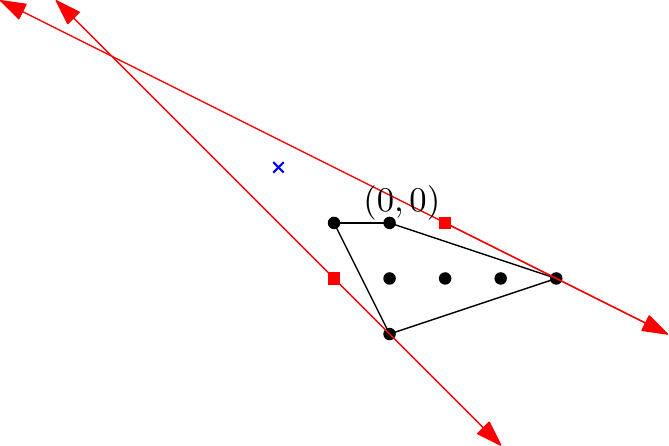}
    \caption{Constraints on where an additional lattice point may be included in Case 2}
    \label{figure:diplo_lines}
\end{figure}


\item  \textbf{Case 3: additional points appear not at heights $0$ and $-2$, but somewhere beyond there.}  Suppose that there are no additional lattice points at heights $0$ and $-2$.  Then the cross-sectional width of $P$ at height $-1$ is at least $3$, and at heights $0$ and $-2$ is strictly smaller than $2$.  It follows that its width at heights $-3$ and $1$ is strictly smaller than $1$, meaning that there is at most one lattice point at such a height, and that there are no lattice points beyond that.  Suppose there were a lattice point at height $1$.  In order not to introduce either of the lattice points $(\pm 1,0)$ via convexity, the only possible points are $(\pm 1,1)$.  Without loss of generality assume it is $(-1,1)$. This combined with the point $(3,-1)$ introduces the point $(0,1)$ by convexity, contradicting the lack of points besides $p$ at height $1$. Thus we find no diploptigons in this case.
\end{itemize}

We have now completed a classification of all diploptigons:  there is the single one of lattice diameter at most $2$ from the previous section, as well as the infinite collection of polygons depicted in Figure \ref{figure:diplo_types_ld3}.  We will now prove that none of these diploptigons can ever be the interior polygon of a lattice polygon.




\begin{proposition}\label{diplolemma}
If $Q$ is a convex lattice polygon, then $Q_{\textrm{int}}$ cannot be a diploptigon.
\end{proposition}

\begin{proof}

Let $P$ be a  diploptigon.  We will argue that $P^{(-1)}$ is not a lattice polygon.  By Lemma \ref{lemma:not_lattice}, this will show that $Q_\textrm{int}$ is not a diploptigon for any lattice polygon \(Q\).


First, if $\ld(P)\leq 2$, then $P$ is the unique diploptigon from Figure \ref{figure:polygons_ld2lw3}. Up to translation, this is defined by the inequalities $-3x-y\leq -3$, $x-2y\leq 1 $, and $2x+3y\leq 9$.  The inequalities defining $P^{(-1)}$ are thus $-3x-y\leq -2$, $x-2y\leq 2 $, and $2x+3y\leq 10$.  The first two inequalities introduce a vertex with coordinates $\left( 6/7,-4/7\right)$ implying that $P^{(-1)}$ is not a lattice polygon.

If $\ld(P)\geq 3$, then $P$ falls into one of the three classes in Figure \ref{figure:diplo_types_ld3}.  In the first case, the top two edges yield the inequalities $-x-\alpha y\leq0$ and $x+(\alpha+k-1)y\leq0$. The relaxed versions of these inequalities are  $-x-\alpha y\leq1$ and $x+(\alpha+k-1)y\leq1$, which yield a vertex at the point $\left(\frac{-2\alpha-k+1}{k-1},\frac{2}{k-1}\right)$.  Since $k\geq 4$, this is not an integral point, so $P^{(-1)}$ is not a lattice polygon.  For the second case bottom two edges are defined by the inequalities $x-ky \leq 2k$ and $-2x-y\leq 2$, which relax to $x-ky \leq 2k+1$ and $-2x-y\leq 3$.  These introduce the intersection point $\left(-\frac{k-1}{2k+1},\frac{-4k-5}{2k+1}\right)$, which cannot be an integer; for instance, $2k+1$ cannot divide $k-1$, since $k\geq 4$.  Finally, the bottom two edges of the third polygon are defined by the inequalities $-3x-2y\leq 4$ and $x-3y\leq 6$.  These relax to $-3x-2y\leq 5$ and $x-3y\leq 7$, yielding an intersection point $\left(-\frac{1}{11},-\frac{26}{11}\right)$.  In every case, $P^{(-1)}$ is not a lattice polygon.  This completes the proof.

\end{proof}

We are now ready to prove that the prism of genus $g$ is tropically planar if and only if $g\leq 11$.

\begin{proof}[Proof of Theorem \ref{thm:prismgraph}, ``only if'' part]  Between Figures \ref{figure:prism_g4_curve} and \ref{figure:prism_curves_5_to_11}, we have examples that illustrate that $\mathfrak{p}_n$ is tropically planar for $3\leq n\leq 10$; that is, for genus $4\leq g\leq 11$.

Conversely, assume that $\mathfrak{p}_n$ is tropically planar.  This means it has a planar embedding dual to a regular unimodular triangulation of a polygon $P$.  If that embedding is the standard embedding of $\mathfrak{p}_n$, then there must be an interior lattice point of $P$ from which all other interior lattice points are visible.  Thus $P_{\textrm{int}}$ is a panoptigon.  By \cite{panoptigons}, the genus of $P$ (and thus of $\mathfrak{p}_n$) is at most $11$.

Conversely, if the embedding is the non-standard embedding, then there must be two interior lattice points $p$ and $q$ of $P$ from which all other interior lattice points of $P$ are simultaneously visible; thus $P_\textrm{int}$ is either a panoptigon (since we do not know if $p$ and $q$ are visible to one another) or a diploptigon.  But it cannot be a diploptigon by Proposition \ref{diplolemma}, and so must be a panoptigon.  Again, we know that the genus is at most $11$.
\end{proof}

\begin{remark}\label{remark:embeddings}
We have already seen in Figure \ref{figure:prism_g4_curve} that the non-standard embedding of the prism graph of genus $4$ can appear in a smooth tropical plane curve.  We might ask, for $6\leq g \leq 11$, whether or not the non-standard embedding can appear in a smooth tropical plane curve (for $g=5$ the prism graph has only one embedding, and for $g\geq 12$ we know there does not exist \emph{any} embedding of the prism in a tropical curve).  Such a tropical curve could not arise from a polygon $P$ where $P_{\textrm{int}}$ is a diploptigon by Proposition \ref{diplolemma}, and so would instead have to arise from a polygon $P$ with $P_{\textrm{int}}$ a panoptigon, and in particular from a triangulation with two points $p$ and $q$ connected to all other points, but not to each other.  An exhaustive search through the panoptigons of genus $6$ through $11$ as classified in \cite{panoptigons} shows that there exists no such triangulation of a polygon whose interior polygon is a panoptigon.  This means the only genus for which the non-standard embedding of the prism graph appears in a tropical curve is $g=4$, giving us the final claim in Theorem \ref{thm:prismgraph}
\end{remark}


\bibliographystyle{plain}

\end{document}